\documentclass[12pt,reqno]{amsart}
\usepackage{amsmath, amsfonts, amssymb, amsthm,mathrsfs}
\def\Z{\mathbb Z}

\def\N{\mathbb N}
\def\Q{\mathbb Q}

\def\cS{\mathcal S}
\def\cC{\mathcal C}

\def\ff{{\mathfrak f}}

\def\1{{\bf 1}}

\def\pmod #1{\ ({\rm{mod}}\ #1)}

\def\maj {{\rm maj}\,}
\def\inv {{\rm inv}\,}
\def\des {{\rm des}\,}
\def\qbinom #1#2#3{{\genfrac{[}{]}{0pt}{}{#1}{#2}}_{#3}}
\theoremstyle{plain}
\newtheorem{theorem}{Theorem}

\newtheorem{lemma}{Lemma}

\theoremstyle{definition}

\theoremstyle{remark}

\begin{document}
\title
{A $q$-analogue of Wilson's congruence}

\author{Hao Pan}
\address {School of Applied Mathematics, Nanjing University of Finance and Economics, Nanjing 210023, People's Republic of China}
\email{haopan79@zoho.com}

\author{Yu-Chen Sun}
\address {Medical School, Nanjing
University, Nanjing 210093, People's Republic of China}
\email{syc@smail.nju.edu.cn}

 \begin{abstract}
 Let $\cC_n$ be the set of all permutation cycles of length $n$ over $\{1,2,\ldots,n\}$.
Let
$$
\ff_n(q):=\sum_{\sigma\in\cC_{n+1}}q^{\maj\sigma}
$$
be a $q$-analogue of the factorial $n!$, where $
\maj$ denotes the major index. We prove a $q$-analogue of Wilson's congruence
$$
\ff_{n-1}(q)\equiv\mu(n)\pmod{\Phi_n(q)},
$$
where $\mu$ denotes the M\"obius function and $\Phi_n(q)$ is the $n$-th cyclotomic polynomial.
\end{abstract}
\maketitle

\section{Introduction}
\setcounter{lemma}{0} \setcounter{theorem}{0}
\setcounter{equation}{0}

For each $n\in\N=\{0,1,2,\ldots\}$, define the $q$-integer
$$
[n]_q:=\frac{1-q^n}{1-q}.
$$
The $q$-integer evidently is a $q$-analogue of the original integer, since
$$
\lim_{q\to 1}[n]_q=\lim_{q\to 1}(1+q+\cdots+q^{n-1})=n.
$$

Suppose that $p$ is a prime. The well-known Lucas congruence says that
\begin{equation}
\binom{ap+b}{cp+d}\equiv\binom{a}{c}\binom{b}{d}\pmod{p}
\end{equation}
where $a,b,c,d\in\N$ and $0\leq b,d\leq p-1$. 
A $q$-analogue of the binomial coefficients
is the $q$-binomial coefficient, which is given by
$$
\qbinom{n}{k}q:=\frac{[n]_q!}{[k]_q![n-k]_q!}
$$
for $n,k\in\N$ with $n\geq k$,
where $$[n]_q!:=[1]_q[2]_q\cdots[n]_q
$$ is the $q$-factorial. Then we have the following $q$-analogue of Lucas' congruence (cf. \cite{Ol65})
\begin{equation}\label{qLucasp}
\qbinom{ap+b}{cp+d}q\equiv\binom{a}{c}\cdot\qbinom{b}{d}q\pmod{[p]_q},
\end{equation}
where the above congruence is considered over the polynomial ring $\Z[q]$.
In fact, (\ref{qLucasp}) can be extended to
\begin{equation}
\qbinom{an+b}{cn+d}q\equiv\binom{a}{c}\cdot\qbinom{b}{d}q\pmod{\Phi_n(q)},
\end{equation}
where 
$$
\Phi_n(q):=\prod_{\substack{1\leq k\leq n\\
(k,n)=1}}(q-e^{2\pi i\cdot\frac{k}{n}})
$$
denotes the $n$-th cyclotomic polynomial.

Another two classical congruences in number theory is the Fermat congruence
\begin{equation}
a^{p-1}\equiv 1\pmod{p}
\end{equation}
for each prime $p$  and $1\leq a\leq p-1$, and the Wilson congruence
\begin{equation}
(p-1)!\equiv -1\pmod{p}
\end{equation}
for each prime $p$. Fermat's congruence also has its $q$-analogue as follows:
\begin{equation}
\prod_{k=1}^{p-1}[a]_{q^k}\equiv 1\pmod{[p]_q}.
\end{equation}
In fact, since $q^a-q^b=q^b(1-q^{a-b})$, 
for each prime $p$,
we have $q^{a}\equiv q^b\pmod{[p]_q}$ if and only if $a\equiv b\pmod{p}$. So, if $p\nmid a$, then
\begin{equation}
\prod_{k=1}^{p-1}[a]_{q^k}=\frac1{\prod_{k=1}^{p-1}(1-q^k)}
\prod_{k=1}^{p-1}(1-q^{ak})\equiv\frac1{\prod_{k=1}^{p-1}(1-q^k)}
\prod_{j=1}^{p-1}(1-q^{j})
=1\pmod{[p]_q}.
\end{equation}
Similarly, we can get
\begin{equation}
\prod_{k=1}^{n-1}[a]_{q^k}\equiv 1\pmod{\Phi_n(q)}
\end{equation}
provided that $(a,n)=1$.

Unfortunately, seemingly there exists no suitable $q$-analogue of Wilson's congruence for the $q$-factorial $[p-1]_q!$. For examples, we have
$$
[6]_q!\equiv 3+3q-4q^3-6q^4-4q^5\pmod{[7]_q}.
$$
Alternatively, in \cite{CP08}, Chapman and Pan gave a partial $q$-analogue of Wilson's congruence for those prime $p>3$ with $p\equiv3\pmod{4}$:
\begin{equation}\label{WilsonCP}
\prod_{k=1}^{p-1}[k]_{q^k}\equiv-1\pmod{[p]_q}.
\end{equation}
However, (\ref{WilsonCP}) is invalid if the prime $p\equiv 1\pmod{4}$, though Chapman and Pan also determined $\prod_{k=1}^{p-1}[k]_{q^k}$ modulo $[p]_q$ for those prime $p\equiv 1\pmod{4}$, with help of the fundamental
unit and the class number of the quadratic field $\Q(\sqrt{p})$.

In this short note, we shall try to obtain a unified $q$-analogue of Wilson's congruence for all primes, from the viewpoint of combinatorics. Let $\cS_n$ denote the permutation group of order $n$, i.e., the set of all permutations over $\{1,2,\ldots,n\}$. 
Clearly $|\cS_n|=n!$.
For each $\sigma\in\cS_n$, define the major index of $\sigma$
$$
\maj\sigma:=\sum_{\substack{1\leq i\leq n-1\\ \sigma(i)>\sigma(i+1)}}i
$$
and the inversion number $\sigma$
$$
\inv\sigma:=\big|\{(i,j):\,1\leq i<j\leq n,\ \sigma(i)>\sigma(j)\}\big|.
$$
It is know (cf. \cite{}) that
\begin{align}
[n]_q!=\sum_{\sigma\in\cS_n}q^{\maj\sigma}
\end{align}
and
\begin{align}
[n]_q!=\sum_{\sigma\in\cS_n}q^{\inv\sigma}.
\end{align}

Let
$$
\cC_n:=\{\sigma\in\cS_n:\,\sigma\text{ is a cycle of length }n\}.
$$
We also have $|\cC_n|=(n-1)!$. Define
\begin{equation}
\ff_n(q):=\sum_{\sigma\in\cC_{n+1}}q^{\maj\sigma}.
\end{equation}
Clearly $\ff_n(q)$ is another $q$-analogue of the factorial $n!$, too. 
In this note, we shall prove a $q$-analogue of Wilson's congruence for $\ff_n(q)$.
Define the M\"obius function
$$
\mu(n):=\begin{cases}
1,&\text{if }n=1,\\
(-1)^k,&\text{if }n=p_1\cdots p_k\text{ where }p_1,\ldots,p_k\text{ are distinct primes},\\
0&\text{if }n>1\text{ is not square-free}.
\end{cases}
$$
\begin{theorem}\label{qWilson} Suppose that $n\geq 2$. Then
\begin{equation}\label{qWilsonne}
\ff_{n-1}(q)\equiv\mu(n)\pmod{\Phi_n(q)}.
\end{equation}
In particular, if $p$ is prime, then
\begin{equation}
\ff_{p-1}(q)\equiv-1\pmod{[p]_q}.
\end{equation}
\end{theorem}
In fact, as we shall see later, the M\"obius function $\mu$ in (\ref{qWilsonne}) arises from the Ramanujan sum
$$
\sum_{\substack{1\leq r\leq n\\ (r,n)=1}}e^{2\pi i\cdot \frac{kr}{n}}.
$$

\section{Proof of Theorem \ref{qWilson}}
\setcounter{lemma}{0} \setcounter{theorem}{0}
\setcounter{equation}{0}

Suppose that $n\geq 2$. Let $\Z_n=\Z/n\Z$ be the cyclic group of order $n$. Below we always identify $\Z_n$ with $\{1,2,\ldots,n\}$, and view $\cS_n$ as the permutation group over $\Z_n$. In particular, for each $1\leq a,b\leq n$, we say $a<b$ over $\Z_n$ if and only if $a<b$ over $\Z$.

Let $\tau\in\cS_n$ be defined by $$\tau(a)=a+1$$ for each $a\in\{1,2,\ldots,n\}$. For each $\sigma\in\cS_n$, let
$$
T\sigma:=\tau\circ\sigma\circ\tau^{-1}.
$$
Then we have $T\cC_n=\cC_n$. In fact, for each $\sigma=(a_1,a_2,\ldots,a_n)\in\cC_n$, 
$$
T\sigma=(a_1+1,a_2+2,\ldots,a_n+1).
$$
Thus $T$ can be viewed as a group action on $\cC_n$, since clearly $T^n\sigma=\sigma$ for each  $\sigma\in\cC_n$. 

For each $\sigma\in\cC_n$, let
$$
U_\sigma:=\{T^k\sigma:\,1\leq k\leq n\}
$$ 
denote the orbit of $\sigma$, where $T^k$ denotes the $k$-th iteration of $T$. Then we may partition $\cC_n$ into
$$
\cC_n=\bigcup_{\sigma\in X}U_\sigma,
$$
where $X\subseteq\cC_n$ satisfies that $\sigma\not\in U_\kappa$ for any distinct $\sigma,\kappa\in X$. 
Since $T$ is a group action, we must have $|U_\sigma|$ divides $n$ for each $\sigma\in X$.
\begin{lemma}\label{Tsigmasigma}
Suppose that $\sigma\in\cC_n$. If $T\sigma=\sigma$, then there exists $1\leq r\leq n-1$ such that
$$
\sigma(a)=a+r
$$
for each $a\in\Z_n$.
\end{lemma}
\begin{proof}
According to the definition of $T$, we have
$$
T\sigma(a+1)=\sigma(a)+1
$$
for each $a\in\Z_n$. Since $T\sigma=\sigma$,
$$
\sigma(a)-a=T\sigma(a+1)-(a+1)=\sigma(a+1)-(a+1)
$$
for each $a\in\Z_n$.
\end{proof}
Let $\lambda_r\in\cS_n$ be the permutation such that $\lambda_r(a)=a+r$ for each $a\in\Z_n$. It is not difficult to verify that $\lambda_r\in\cC_n$ if and only if $(r,n)=1$. According to Lemma \ref{Tsigmasigma}, for each $\sigma\in X$, $|U_\sigma|=1$ if and only if $\sigma=\lambda_r$ for some $r$ prime to $n$. That is,
\begin{equation}\label{partitionCn}
\cC_n=\{\lambda_r:\,1\leq r\leq n,\ (r,n)=1\}\cup\bigcup_{\substack{\sigma\in X\\ |U_\sigma|>1}}U_\sigma.
\end{equation}

Recall that
$$
\maj\sigma=\sum_{\substack{1\leq i\leq n-1\\ \sigma(i)>\sigma(i+1)}}i.
$$
Define
$$
\overline{\maj}\sigma=\sum_{\substack{1\leq i\leq n\\ \sigma(i)>\sigma(i+1)}}i.
$$
Notice that
$$
\overline{\maj}\sigma=\begin{cases}\maj\sigma+n,&\text{if }\sigma(n)>\sigma(1),\\
\maj\sigma,&\text{otherwise}.
\end{cases}
$$
So we always have
$$
\overline{\maj}\sigma\equiv\maj\sigma\pmod{n}.
$$
\begin{lemma}\label{majdesT}
Suppose that $\sigma\in\cC_n$. Then
\begin{equation}\label{majTmajdes}
\overline{\maj} T\sigma\equiv\overline{\maj} \sigma+\overline{\des}\sigma-1\pmod{n},
\end{equation}
where
$$
\overline{\des}\sigma=\sum_{\substack{1\leq i\leq n\\ \sigma(i)>\sigma(i+1)}}1.
$$
Furthermore,
\begin{equation}\label{desTdes}
\overline{\des} T\sigma=\overline{\des}\sigma.
\end{equation}
\end{lemma}
\begin{proof}
If $\sigma(i-1),\sigma(i)\neq n$, then
$$
T\sigma(i)>T\sigma(i+1)\Longleftrightarrow\sigma(i-1)+1
>\sigma(i)+1\Longleftrightarrow\sigma(i-1)>\sigma(i).
$$
Assume that $\sigma(i_0)=n$.
Clearly $$
T\sigma(i_0+1)=n+1=1<T\sigma(i_0+2),
$$ 
as well as $T\sigma(i_0)>T\sigma(i_0+1)$. 
Hence
$$
\overline{\maj} T\sigma=i_0+\sum_{\substack{1\leq i\leq n\\
i\neq i_0,i_0+1\\
T\sigma(i)>T\sigma(i+1)}}i=i_0+\sum_{\substack{1\leq i\leq n\\
i\neq i_0,i_0+1\\
\sigma(i-1)>\sigma(i)}}i,
$$
where we identify $\sigma(0)$ with $\sigma(n)$. Apparently
$$
\sum_{\substack{1\leq i\leq n\\
i\neq i_0,i_0+1\\
\sigma(i-1)>\sigma(i)}}i=
\sum_{\substack{0\leq i\leq n-1\\
i\neq i_0-1,i_0\\
\sigma(i)>\sigma(i+1)}}(i+1)\equiv\sum_{\substack{1\leq i\leq n\\
i\neq i_0-1,i_0\\
\sigma(i)>\sigma(i+1)}}(i+1)\pmod{n}.
$$
It follows that
$$
\overline{\maj} T\sigma\equiv
i_0+\sum_{\substack{1\leq i\leq n\\
i\neq i_0-1,i_0\\
\sigma(i)>\sigma(i+1)}}i+\sum_{\substack{1\leq i\leq n\\
i\neq i_0-1\\
\sigma(i)>\sigma(i+1)}}1\pmod{n}.
$$
Finally, since $\sigma(i_0)=n$ is greater than $\sigma(i_0-1)$ and $\sigma(i_0+1)$, we have
$$
i_0+\sum_{\substack{1\leq i\leq n\\
i\neq i_0-1,i_0\\
\sigma(i)>\sigma(i+1)}}i=\overline{\maj}\sigma
$$
and
$$
\sum_{\substack{1\leq i\leq n\\
i\neq i_0-1,i_0\\
\sigma(i)>\sigma(i+1)}}1=\overline{\des}\sigma-1.
$$
(\ref{majTmajdes}) is concluded.

Similarly, we also have
$$
\overline{\des} T\sigma=1+\sum_{\substack{1\leq i\leq n\\
i\neq i_0,i_0+1\\
T\sigma(i)>T\sigma(i+1)}}1=1+\sum_{\substack{1\leq i\leq n\\
i\neq i_0-1,i_0\\
\sigma(i)>\sigma(i+1)}}=\overline{\des} \sigma.
$$
\end{proof}

\begin{lemma}\label{dessigma1} For each $\sigma\in\cC_n$,
$\overline{\des}\sigma=1$ if and only if $\sigma=\lambda_r$ for some $r$ prime to $n$.
\end{lemma}
\begin{proof}
Assume that $\sigma(i_0)=n$. Clearly $i_0\neq n$, otherwise $\sigma\not\in\cC_n$. Since $\overline{\des}\sigma=1$ and $\sigma(i_0)=n>\sigma(i_0+1)$, we must have 
$$
\sigma(i_0+1)<\sigma(i_0+2)<\cdots<\sigma(n)<\sigma(1)<\sigma(2)<\cdots<\sigma(i_0-1).
$$
Since $\sigma(i_0+1)\geq 1$ and $\sigma(i_0-1)\leq b=n-1$,
we obtain that $\sigma(i_0+j)=j$ for each $1\leq j\leq n-i_0$ and $\sigma(k)=n-i_0+k$ for each $1\leq k\leq i_0-1$, i.e., $\sigma=\lambda_{n-i_0}$.
\end{proof}

Now we are ready to prove Theorem \ref{qWilson}. 
In view of (\ref{partitionCn}),
$$
\ff_{n-1}(q)=
\sum_{\substack{1\leq r\leq n\\ (r,n)=1}}q^{\maj\lambda_r}+
\sum_{\substack{\sigma\in X\\ |U_\sigma|>1}}\sum_{\upsilon\in U_\sigma}q^{\maj \upsilon}.
$$
Suppose that $\sigma\in X$ and $|U_\sigma|\geq 2$. By Lemma \ref{dessigma1}, we have $\overline{\des}\sigma\geq 2$. Let $h=|U_\sigma|$. According to Lemma \ref{majdesT},
\begin{align*}
\sum_{\upsilon\in U_\sigma}q^{\maj \upsilon}\equiv\sum_{k=0}^{h-1}q^{\overline{\maj} T^k\sigma}\equiv&
q^{\overline{\maj} \sigma}\sum_{k=0}^{h-1}q^{k(\overline{\des}\sigma-1)}\\
=&
q^{\overline{\maj} \sigma}\cdot\frac{1-q^{h(\overline{\des}\sigma-1)}}{1-q^{\overline{\des}\sigma-1}}\pmod{\Phi_n(q)}.
\end{align*}
Since $1\leq \overline{\des}\sigma-1\leq n-1$, $1-q^{\overline{\des}\sigma-1}$ is not divisible by $\Phi_n(q)$. On the other hand, $T^h\sigma=\sigma$ since  $|U_\sigma|=h$. So, by Lemma \ref{majdesT}, we must have
$$
h(\overline{\des}\sigma-1)\equiv0\pmod{n},
$$
i.e.,
$$
1-q^{h(\overline{\des}\sigma-1)}\equiv0\pmod{\Phi_n(q)}.
$$
Thus for each $\sigma\in X$ with $|U_\sigma|>1$, we have
$$
\sum_{\upsilon\in U_\sigma}q^{\maj \upsilon}\equiv0\pmod{\Phi_n(q)}.
$$
It follows that
$$
\ff_{n-1}(q)\equiv
\sum_{\substack{1\leq r\leq n\\ (r,n)=1}}q^{\maj\lambda_r}=
\sum_{\substack{1\leq r\leq n\\ (r,n)=1}}q^{n-r}=\sum_{\substack{1\leq r\leq n\\ (r,n)=1}}q^r\pmod{\Phi_n(q)}.
$$

Finally, it suffices to show that
$$
\sum_{\substack{1\leq r\leq n\\ (r,n)=1}}q^r\equiv\mu(n)\pmod{\Phi_n(q)}.
$$
Let $\zeta$ be a $n$-th primitive root of unity. Then
$$
\Phi_n(q)=\prod_{\substack{1\leq k\leq n\\ (k,n)=1}}(q-\zeta^k).
$$
So we only need to prove that for each $1\leq k\leq n$ with $(k,n)=1$,
$$
\lim_{q\to\zeta^k}\sum_{\substack{1\leq r\leq n\\ (r,n)=1}}q^r=\mu(n),
$$
which immediately follows from our knowledge of Ramanujan's sum.
All are done.\qed

\end{document}